\documentclass[a4paper,10pt]{article}
\usepackage[english]{babel}
\usepackage[cp1251]{inputenc}
\usepackage[left=44mm,right=40mm,top=60mm,bottom=52mm]{geometry}

\usepackage{graphicx}
\usepackage{epsfig}
\usepackage{amsthm}
\usepackage{amsmath}
\usepackage{amsfonts}
\usepackage{amssymb}
\usepackage{amscd}

\newtheorem{theorem}{Theorem}

\newtheorem{corollary}{Corollary}
\newtheorem{definition}{Definition}
\newtheorem{lemma}{Lemma}

\usepackage{mathrsfs} 

\usepackage{color}

\newcommand{\bb}{\color{blue}}

\newcommand{\dd}{\ensuremath{\displaystyle}}
\newcommand{\liml}{\ensuremath{\lim\limits}}
\newcommand{\infl}{\ensuremath{\inf\limits}}
\newcommand{\supl}{\ensuremath{\sup\limits}}
\newcommand\intl{\ensuremath{\int\limits}}
\newcommand{\suml}{\ensuremath{\sum\limits}}
\newcommand{\PP}{\ensuremath{\mathbf{P}}}

\newcommand{\PPP}{\ensuremath{\mathcal{P}}}
\newcommand{\XX}{\ensuremath{\mathcal{X}}}
\newcommand{\SSS}{\ensuremath{\mathcal{S}}}

\newcommand{\EE}{\mathsf E}
\newcommand{\bD}{\stackrel{\mathcal{D}}{=}}
\newcommand{\1}{\ensuremath{\mathbf{1}}}

\newcommand{\ud}{\,\mathrm{d}}
\newcommand*{\TR}{\hfill\ensuremath{\triangleright}}
\newcommand{\bd}{\stackrel{{\rm def}}{=\!\!\!=}}

\begin{document}
\section*{Lorden's inequality and the polynomial rate of convergence of some extended Erlang-Sevastyanov queuing system\thanks{The work is supported by RFBR, project No~20-01-00575 A.}}
\begin{flushright}
\textbf{Galina Zverkina\\
{\it V. A. Trapeznikov Institute of Control Sciences of Russian Academy of Sciences, 65 Profsoyuznaya street, Moscow 117997, Russia} }
\end{flushright}

\begin{abstract}
It is more important to estimate the rate of convergence to a stationary distribution rather than only to prove the existence one in many applied problems of reliability and queuing theory. This can be done via standard methods, but only under assumptions about an exponential distribution of service time, independent intervals between recovery times, etc. Results for such simplest cases are well-known. 
Rejection of these assumptions results to rather complex stochastic processes that cannot be studied using standard algorithms. 
A more sophisticated approach is needed for such processes. That requires generalizations and proofs of some classical results for a more general case. One of them is the generalized Lorden's inequality proved in this paper.
We propose the generalized version of this inequality for the case of dependent and arbitrarily distributed intervals between recovery times. This generalization allows to find upper bounds for the rate of convergence for a wide class of complicated processes arising in the theory of reliability. 
The rate of convergence for a two-component process has been obtained via the generalized Lorden's inequality  in this paper.

\noindent{\bf keywords} {\it Lorden's inequality, Convergence rate, Strong upper bounds, Coupling method, Successful coupling.}
\end{abstract}
\section{Introduction}
Let's consider an {\it extended} $M|G|\infty$ queueing system, where intensities of input flow and service depend on the {\it full state of the system $X_t$} (the full state of the system $X_t$ is described below) -- see Fig.1.
The study of the behavior of such a system 
with arbitrary dependencies between parameters
is impossible
via a standard technique.
\begin{center}
\begin{picture}(100,210)
\thicklines
\put(0,75){\vector(1,0){30}}
\put(0,80){$\lambda(X_t)$}
\put(40,0){\line(1,0){20}}
\put(40,20){\line(1,0){20}}
\put(40,30){\line(1,0){20}}
\put(40,50){\line(1,0){20}}
\put(40,60){\line(1,0){20}}
\put(40,80){\line(1,0){20}}
\put(40,90){\line(1,0){20}}
\put(40,110){\line(1,0){20}}
\put(40,120){\line(1,0){20}}
\put(40,140){\line(1,0){20}}
\put(40,150){\line(1,0){20}}
\put(40,170){\line(1,0){20}}
\put(40,0){\line(0,1){20}}
\put(60,0){\line(0,1){20}}
\put(40,30){\line(0,1){20}}
\put(60,30){\line(0,1){20}}
\put(40,60){\line(0,1){20}}
\put(60,60){\line(0,1){20}}
\put(40,90){\line(0,1){20}}
\put(60,90){\line(0,1){20}}
\put(40,120){\line(0,1){20}}
\put(60,120){\line(0,1){20}}
\put(40,150){\line(0,1){20}}
\put(60,150){\line(0,1){20}}
\put(50,175){\circle*{3}}
\put(50,180){\circle*{3}}
\put(50,185){\circle*{3}}
\put(50,190){\circle*{3}}
\put(50,195){\circle*{3}}
\put(50,200){\circle*{3}}
\put(50,10){\circle*{15}}
\put(50,40){\circle*{15}}
\put(50,70){\circle*{15}}
\put(50,100){\circle*{15}}
\put(70,10){\vector(1,0){30}}
\put(70,17){$h_1(X_t)$}
\put(70,100){\vector(1,0){30}}
\put(70,107){$h_n(X_t)$}
\put(70,40){\vector(1,0){30}}
\put(70,47){$h_2(X_t)$}
\put(80,57){\circle*{3}}
\put(80,62){\circle*{3}}
\put(80,67){\circle*{3}}
\put(80,72){\circle*{3}}
\put(80,77){\circle*{3}}
\put(80,82){\circle*{3}}
\put(80,87){\circle*{3}}
\put(80,92){\circle*{3}}
\put(80,97){\circle*{3}}
\end{picture}
\\
\vspace{0.5cm}
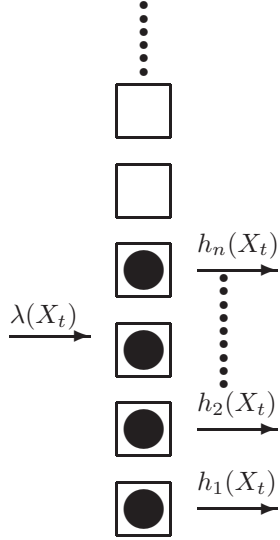
Fig.1. Studied extended Erlang-Sevastyanov Queuing System
\end{center}

\subsection{Full state of the system $X_t$}

The full system state at the time $t$ is described by the following vector 
$$X_t=\left(n_t,
x^{(0)}_t; x^{(1)}_t,x^{(2)}_t,\ldots,x^{(n_t)}_t\right),$$ where

$x^{(0)}_t$ -- the elapsed time from the last customer input, 

$x^{(i)}_t$ -- the elapsed times of the service of all customers staying in the system,

$n_t$ -- the number of the customers in the system at the time $t$. 

We also use suppose for a 
simplicity that $x^{(i)}_t\ge x^{(i+1)}_t$ for $i=1,\ldots,n_t$,

It is easy to see that $x^{(0)}_t\le x^{(n_t)}_t$, i.e. the customer
numbered by $n_t$ appears in the system no later than at the time
$t-x^{(0)}_t$.

Suppose the intensity of input flow is $\lambda=\lambda(X_t)$, and intensity of service for $i$-th customer is $h_i=h_i(X_t)$.

Let 
$$
X_0=(0;0)
$$
for simplicity.
\subsection{Conditions}\label{cond}

\begin{enumerate}
\item There exists the positive function $\lambda_0(s)$ and the positive constant $\Lambda$ such that for all possible $X_t$,  $\lambda_0\left(x_t^{(0)}\right)\le  \lambda(X_t) \le \Lambda<\infty$.

~

\item There exist two a.s. positive (generalized) functions $\varphi(t)$ and $Q(t)$ ($Q(t)$ is bounded in a neighborhood of zero) such that for all \linebreak $X_t=\left(n_t, x^{(0)}_t; x^{(1)}_t,x^{(2)}_t,\ldots,x^{(n_t)}_t\right)$ the following inequalities hold :

$\varphi\left(x_t^{(i)}\right)\le  h_i(X_t)\le  Q\left(x_t^{(i)}\right)$ ~~ --- ~~ for $i=1,2,\ldots n_t$.

~

\item $\dd \intl_0^\infty x^k\ud \Phi(x)< \infty$ for $k\ge 2$, where $\Phi(x)=1-\dd\intl_0^x\exp\left(\intl_0^s -\varphi(s)\ud s\right)\ud x$.

\end{enumerate}

\subsection {Intensities}
\begin{definition}{About intensities}.

If some object lives a random time $\tau$ with the distribution function $G(s)$, and it is alive at the time $t$, then the probability of its dead in the interval $(t,t+\Delta)$ is equal
$$
\dd \frac{G(t+\Delta)-G(t)}{1-G(t)}= \dd \frac{G'(t+0)}{1-G(t)}\Delta+ o(\Delta); \qquad h(t) \stackrel{{\rm def}}{=\!\!\!=} \dd \frac{G'(t+0)}{1-G(t)};
$$
and $G(t)=\dd 1-\exp\left(\dd-\intl_0^t h(s)\ud s\right)$, if $G(t)$ is absolutely continuous.

If the distribution function $G(s)$ has a jump at the point $a$, then it's intensity at this jump point $a$ is equal $-\delta(x-a)\times \ln (G(a+0)-G(a-0)$, where $\delta(\ast)$ is a classic $\delta$-function.
\TR
\end{definition}

So, 
$$\varphi(s)\bd \dd\frac{G'(s)}{1-G(s)}-\suml_{i}\delta(s-a_i)\ln\big(G(a_i+0)-G(a_i-0) \big),$$ where $\{a_i\}$ --- is the set of all points of discontinuity of a function 
$G(s)$.

And the classic formula $G(t)=\dd 1-\exp\left(\dd-\intl_0^t h(s)\ud s\right)$ is true.

{\it Thus, mixed distribution functions of service time are admitted.}

The generalized intensities define a Markov process $X_t$.

Our aim is to prove the ergodicity of the process $X_t$ and to find upper bounds for convergence rate of it's distribution to the limit one.

\section {Ergodicity of the process $X_t$.}
Note, $X_t$ is regenerative process: regeneration points are points of new customer arrivals into idle system, i.e. $X_t=(1;0;0)$ (before these points of time the system is empty).

Let $R_1,$ $R_2,$ $R_3,$ be \ldots lengths of sequential regeneration periods.
These random variables are i.i.d., they consist of the idle period $\sigma_i$ and busy period $\zeta_i$; $\EE\,\sigma_i^k\le \dd \frac{k!}{\lambda_0^k}$.

Moreover, the condition 3 from {\bf Introduction} holds: $\EE\,(\sigma_i+\zeta_i)<\infty$. 
Thus, the process $X_t$ is regenerative and ergodic.
So, it's distribution $\mathcal P_t$ at the time $t$ converges to the invariant limit distribution  $\mathcal P$:  $\mathcal P_t \Longrightarrow \mathcal P$.

\begin{theorem}[Well-known fact]
It is well-known, that if $\EE\,(\sigma_i+\zeta_i)^k<\infty$ for some $k>1$, then the rate of convergence of the distribution of the regenerative process -- to the stationary distribution -- is less than $\dd\frac{\mathbf{K}}{ t^{k-1}}$ for {\it some} $\mathbf{K}$.
\end{theorem}
If we don't know the constant $\mathbf{K}$, we can not use this fact in practice.
\begin{center}
{\bf Our goal is to find upper bounds for this constant $\mathbf{K}$.}
\end{center}

\subsection {Theorem and corollary}
\begin{theorem}[See \cite{kalzv}]\label{thm1} 
If the conditions 1--3 are satisfied and $\EE\,\eta^k<\infty$, then the following inequality for the process $B_t$ holds: 
\begin{equation}\label{osn}
\EE\,(b_t)^{k-1}\le \EE\, \eta ^{k-1} + \frac{\EE\,\eta^k}{k\EE\,\zeta},
\end{equation}
where
$ \EE\, \eta^{k}=\dd \intl_0^\infty x^{k} \ud \Phi(x);\qquad
\EE\, \zeta= \intl_0^\infty x \ud G(x); \mbox{~~and~~}
G(x)= 1- \intl_0^x \exp ^{-\intl_0^s Q(t)\ud t} \ud s.\qquad
$
\end{theorem}
\begin{corollary}
If conditions 1--3 are satisfied, then
$$
\EE\,(b_t)\le \EE\, \eta + \frac{\EE\,\eta^2}{2\EE\,\zeta}.
$$
\TR
\end{corollary}

\section{Comparison of the process $X_t$ with the ``standard'' process for $M|G|\infty$}

Let's consider the ``standard'' queueing system $M|G|\infty$ that has 
the constant intensity $\Lambda$ of input flow.  Service times are
independent identically distributed random variables with a distribution function $\Phi_0(s)=1-\dd \intl_0^x \exp ^{-\intl_0^s \varphi(t)\ud t} \ud s$. 

The time intervals between arrivals of customers and service times are mutually independent.

Following the scheme described previously we can construct the Markov process $Y_t=\left(m_t; y^{(0)}_t; y^{(1)}_t,y^{(2)}_t,\ldots,y^{(n_t)}_t\right)$ for this system, where $m_t$ is the number of customers in the system at time $t$, $y^{(0)}_t$ is the elapsed time from the
last customer arrival, and $y^{(i)}_t$ is the elapsed time of service for the $i$-th customer, currently being in the system; $Y_0=(0;0)$.

\begin{lemma}  The processes $X_t$ and $Y_t$ on the same probability space
are defined by the following way:
\begin{enumerate}

\item $X_0=Y_0=(0;0)$;

~

\item for all $t\ge 0$, $n_t\le m_t$;

~

\item for all $t\ge 0$ there are the numbers $k_1<k_2<\cdots<k_{n_t}\le m_t$ such that $x^{(i)}_t\le y^{(k_i)}_t$.

\end{enumerate}
\end{lemma}
Proof is based on the facts from \cite{sht}.
\begin{definition}
We say that $\xi\prec \eta$ if $\PP\{\xi<s\}\ge \PP\{\eta<s\}$ for all $s\in\mathbf      R$.
\end{definition}

So, the busy period $\zeta^X$ of $X_t$ is less then busy period
$\zeta^Y$ of $Y_t$ in the sense of this definition: $\zeta^X\prec
\zeta^Y$; moreover, $\PP\{n_t=0\}\ge \PP\{m_t=0\}$, and
$\EE\,\left[ \zeta^X\right]^k\le \EE\,\left[ \zeta^Y\right]^k$
for all $k>0$.

\subsection  {Well-known facts about ``standard'' system $M|G|\infty$ -- see \cite{Takac,Stad,Fere}}
In a ``standard'' system $M|G|\infty$, the intensity of input flow is constant $\Lambda$, the service times $\xi_i$ are i.i.d.r.v. with distribution function $$\PP\{\xi_i<s\}=\Phi_0(s)= 1- \intl_0^x \exp ^{-\intl_0^s \varphi(t)\ud t} \ud s.$$

For this system (if $X_0=(0;0)$):
\begin{itemize}
\item $\dd \PP\{m_t=k\}=\frac{e^{-\mathcal G(t)}\big(\mathcal G(t)\big)^k}{k!},$ where $\mathcal G(t)\bd\dd \Lambda\intl_0^t(1-\Phi_0(s))\ud s$.

So, for all $t\ge 0$,
$\dd
\PP\{m_t=0\}\ge \liml_{t\to \infty}\PP\{m_t=0\}=e^{-\rho},
$
where 
\begin{equation}\label{rho}
\rho\bd \dd \Lambda m_1=\Lambda\dd \intl_0^ \infty
(1-\Phi_0(s))\ud s.
\end{equation}

\item For the busy period of the considered system $M|G|\infty $
(denote the busy periods by $\zeta_i$) we know the distribution
function
\begin{equation}\label{tak2}
B(x)\bd \PP\{\zeta_i\le x\}=1-
\frac{1}{\Lambda}\suml_{k=1}^\infty c^{n\ast}(x),
\end{equation}
where $c(x)\bd \dd \Lambda(1-\Phi_0(x))e^{-\mathcal G (x)}$, and $c^{n\ast}$ is a
$n$-th convolution of the function $c(x)$.

\item
Moreover, we know the Laplace transform of $B(x)$:
\begin{equation}\label{lap}
\mathcal{L}[B](s)=1+\frac{s}{\Lambda}-\frac{1}{\Lambda \dd\intl_0^\infty e^{-st-\Lambda\intl_0^t [1-\Phi_0(v)]\ud v}\ud t}.
\end{equation}

\item Also we have formulae:
\begin{equation}\label{razl}
\EE\,\zeta_i^n=(-1)^{n+1}\left\{\frac{e^\rho}{\Lambda}nC^{(n-1)} - e^\rho \suml_{k=1}^{n-1}C_n^k \EE\,\zeta_i^{n-k} C^{(p)} \right\},
\qquad
n\in\mathbf     {N}, 
\end{equation}
where
$$
C^{(n)}=\intl_0^\infty (-t)^n \left(e^{-\Lambda
\intl_0^\infty [1- \Phi_0(v)]\ud v}\right)\Lambda
[1- \Phi_0(t)]\ud t.
$$
\end{itemize}

From formulae (\ref{tak2})--(\ref{razl}), we have 
$$\EE \,
\zeta_i=\dd\frac{e^\rho-1}{\Lambda}$$ 
and 
$$\EE
\,\zeta_i^2=\dd\frac{2e^{2\rho}}{\Lambda}\dd\intl_0^\infty \left(e^{-\Lambda
\intl_0^\infty [1-\Phi_0(v)]\ud v}\right)\ud t.
$$

But for the next moments of $\zeta_i$ the calculations are very
complicated, and we want to {\it estimate} this moments.

\subsection{Estimation}
Return to the formula
$\dd
B(x)\bd \PP\{\zeta_i\le x\}=1-
\frac{1}{\Lambda}\suml_{k=1}^\infty c^{n\ast}(x),
$
where $c(x)\bd \dd \Lambda(1- \Phi_0(x))e^{-\mathcal G (x)}$, and $c^{n\ast}$ is a
$n$-th convolution of the function $c(x)$.

~

It is easy, $\dd\intl_0^\infty c(s)\ud s=1-e^{-\Lambda m_1}=1-e^{-\rho}\bd \varrho\in(0;1)$.

So, $r(s)\bd \dd \frac{c(s)}{ \varrho}$ is a density of distribution of some r.v. $\varsigma$.
And $c^{n\ast}(x)=\varrho^n r_n(x)$ where $ r_n(x)$ is a density of distribution of the sum $\dd\suml_{i=1}^n \varsigma_i$ where $\varsigma_i$ are i.i.d.r.v. with distribution density $r(s)$.

Further, 
$$\dd\EE\,\varsigma^k=\intl_0^\infty s^k
\varrho^{-1}\Lambda(1- \Phi_0(s))e^{-\mathcal G (s)}\ud s\le 
\frac{\Lambda}{ \varrho}\intl_0^\infty s^k (1- \Phi_0(s))\ud
s=\frac{\Lambda\EE\,\xi_i^{k+1}}{(k+1) \varrho}.
$$

Now, using the Jensen's inequality in the form
$$\dd (a_1+\ldots+a_n)^k\le n^{k-1}(a_1^k+\ldots +a_n^k),
$$ 
we have:
\begin{eqnarray*}
\EE\,\zeta_i^k=\intl_0 ^ \infty k x^{k-1}(1-B(x))\ud x =\frac{1}{\Lambda}\suml_{n=1}^\infty\intl_0^ \infty k x^{k-1}
c^{n\ast}(x)\ud x =\qquad \qquad \qquad \quad
\\ \\
= \frac{1}{\Lambda}\suml_{n=1}^\infty\intl_0^
\infty k x^{k-1} \varrho^n r^{n\ast}(x)\ud x=
\frac{1}{\Lambda}\suml_{n=1}^\infty \varrho^n k\,
\EE\left(\suml_{j=1}^n \varsigma_j\right)^ {k-1}\le  \qquad \qquad
\\ \\
\le \frac{1}{ \Lambda} \suml_{n=1}^\infty \varrho^nk
n^{k-1}\frac{\Lambda\EE\, \xi_i^k}{\varrho k}= \frac{\EE\,\xi_i
^k}{\varrho } \suml_{n=1}^\infty n^{k-1} \varrho^n= \frac{\EE\,\xi_i
^k}{\varrho } \times \varphi(\varrho,k-1),
\end{eqnarray*}
where $\dd \varphi(x,k)\bd \left(x\,\frac{\ud}{\ud
x}\right)^k\frac{1}{1-x}$.
We will use the {\it coupling method} to obtain upper bounds for a convergence rate.

\subsection{Coupling method}

{\it If two homogeneous Markov processes  $X_t$ and $X_t'$ with the same transition function but with different initial states coincide at the time $\tau$, then  their distributions are equal after the time $\tau$.}
So, 
\begin{eqnarray*}|\PP\{X_t\in S\}-\PP\{X_t'\in S\}|=\hspace{7cm}
\\ \\
=|\PP\{X_t\in S\}-\PP\{X_t'\in S\}|\times (\1(\tau> t)+ \1(\tau\le t))\le \PP\{\tau> t\}= \qquad
\\ \\
=\PP\{\tau^k> t^k\}\le \dd\frac{\EE\,\tau^k}{t^k}.
\end{eqnarray*}
The next modification of the coupling method is used for the processes in continuous time.

\subsubsection{Successful coupling (see \cite{Griff}}
So, we will create \emph{on some probability space} the paired process $\mathcal{Z}_t=(Z_t,Z_t')$ such that
\begin{itemize}
\item[(i)] For all $s\ge  0$ and $S\in \mathcal B ( \mathcal X)$, ~
$
\mathbf{P} \{Z_t\in S\}= \mathbf{P} \{X_t\in S\}$, $ \mathbf{P} \{Z'_t\in S\}= \mathbf{P} \{X'_t\in S\}.
$
(Therefore, $Z_0=X_0$ and $Z'_0=X'_0$.)

~

\item[(ii)] For all $t>\tau(X_0,X'_0)=\tau(Z_0,Z'_0) \bd \inf\{t\ge  0:\,Z_t=Z_t'\}$ the equality $Z_t=Z_t'$ is true.

~

\item[(iii)] For all $X_0,X_0'\in \XX$, $ \mathbf{P} \{\tau(X_0,X'_0)<\infty\}=1$.
\end{itemize}
If the conditions {\bb(i)}--{\bb(iii)} are satisfied, then the paired process $\mathcal{Z}_t$ is called \emph{successful coupling}.
\begin{eqnarray*}
\mbox{Hence, }\;|\PP\{X_t\in S\}-\PP\{X_t'\in S\}|=|\PP\{Z_t\in S\}- \PP\{Z_t'\in S\}|\le \qquad \qquad \qquad 
\\ \\
\le \PP\{\tau>t\}\le \PP\{\phi(\tau)>\phi(t)\}\le \frac{\EE\,\phi(\tau)}{\phi(\tau)}=\mathcal{R}(t,X_0,X_0'),\qquad \phi\uparrow,\quad \phi>0.
\end{eqnarray*}
And $$\|\mathcal{P}_t-\mathcal{P}_t'\|_{TV}\bd2\supl_{S\in \mathcal{B}(\mathcal{X})}|\PP\{X_t\in S\}-\PP\{X_t'\in S\}|\le 2\mathcal{R}(t,X_0,X_0').$$

{Return to our studied ({see \bf Introduction}) process $X_t$ and its variant $X_t'$}.
Let $X_0=(0;0)$, and $X_0'$ is arbitrary from $\XX$; 
$$X_t=\left(n_t,
x^{(0)}_t; x^{(1)}_t,x^{(2)}_t,\ldots,x^{(n_t)}_t\right);$$ 
$$X_t'=\left(n_t',
{x^{(0)}_t}';{ x^{(1)}_t}',{x^{(2)}_t}',\ldots,{x^{(n_t')}_t}'\right).$$

Consider the times $\theta_1$, $\theta_2$, $\theta_3$,\ldots when $X_t'$ comes to the set $\SSS_1$, or $n_t'$ changes the value from 1 to 0.
$$
\PP\{X_{\theta_i}\in \SSS_1\}=\PP\{n_{\theta_i}=0\}\ge \PP\{m_{\theta_i}=0\}\ge  e^{-\rho}\bd \pi_0.
$$
If $n_{\theta_i+}=n'_{\theta_i+}$ then both systems are idle with probability greater than $\pi_0$. 
Now, we apply the Coupling Lemma and Generalized Lorden's inequality.

\subsection{Coupling Lemma}
\begin{lemma}[Coupling Lemma -- see, e.g., \cite{kato,verbut}]
Let $f_i(s)$ be the distribution density of r.v. $\theta_i$ ($i=1,2$).
And let $\dd\intl_{-\infty}^ \infty \min(f_1(s),f_2(s))\ud s=\tilde \kappa >0.$
Then on some probability space there exists two random variables $\vartheta_i$ such that $\vartheta_i\bD \theta_i$, and $\PP\{\vartheta_1=\vartheta_2\}\ge  \tilde \kappa $.
\end{lemma}
\begin{proof} It is constructive -- see, e.g., \cite{kato} and \cite{Chang}.
\end{proof}

\subsection {Generalized Lorden's inequality (\cite{Chang,Lorden}). }
Let's consider a countable process with jumps $N_t\bd\dd\suml_{i=1}^\infty \1\left\{ \suml_{k=1}^i \xi_k\le t\right\} $, where
$\left\{\xi_1, \xi_2, ...\right\} $ are independent identically distributed (i.i.d.) positive random variables.
Consider the {\it backward renewal time} (or overshoot) for this process:
$
\dd b_t=t-\sum_{k=1}^{N_t} \xi _k.
$

Let us consider this counting process, where $\xi_j$ -- r.v., that may be dependent.

Let $\PP\{\xi_j\le s\}=F_j(s)$; $F_j$ and $F_i$ may not be equal, but the conditions 1--3 are satisfied; $\Phi(x)\ge F_j(x)\ge \Psi(x)$ -- see above.
\begin{theorem}\label{thm2} 
If the conditions 1--3 for intensities of $\xi_i$ are satisfied and $\EE\,\eta^k<\infty$, then the following inequality for the process $B_t$ holds: 
\begin{equation}\label{osn1}
{{\EE\,(b_t)^{k-1}\le  \EE\,  \eta ^{k-1} +  \frac{\EE\,\eta^k}{k\EE\,\zeta},}}
\end{equation}
where
$ \EE\,  \eta^{k}=  \intl_0^\infty x^{k} \ud \Phi(x);\qquad
 \EE\,  \zeta= \intl_0^\infty x \ud \Psi(x); \mbox{~~and~~}
\Psi(x)= 1- \intl_0^x \exp ^{-\intl_0^s Q(t)\ud t} \ud s.\qquad
  $
\end{theorem}
\begin{corollary}
In conditions 1--3 are satisfied, then
$$
\EE\,(b_t)\le  \EE\,  \eta +  \frac{\EE\,\eta^2}{2\EE\,\zeta}.\eqno\TR
$$
\end{corollary}

If one of two processes $X_t$, $X_t'$ hits to the set $\mathcal{S}_0\bd \{(0,\alpha)\}$, then with probability greater then $\pi_0=\rho$ (see (\ref{rho})) the second process is in idle state.

The elapsed time of the idle period of the second process has an expectation $\EE\,b\le \EE\, \eta + \dd\frac{\EE\,\eta^2}{2\EE\,\zeta}\bd \Theta_0$ (Generalized Lorden's inequality).

Therefore, for any $\Theta>\Theta_0$, $\PP\{b<\Theta\}\ge  1-\dd\frac {\Theta_0}{\Theta}$ by Markov inequality.

Both intensities $\varphi>0$ for $X_t=\left(0,x_t^{(0)}\right)$, $X_t'=\left(0,{x_t'}^{(0)}\right)$, therefore 
$$\kappa\bd \infl_{\alpha\le b}\intl_0^\infty \min\{\lambda_0,(\alpha+s),\lambda(0,s)\}\ud s=\pi_1.$$

Hence, with probability greater then $\pi_0\pi_1$ we can apply basic coupling lemma, and after coincidence of both processes, their distributions will be equal.

Thus, after any hit of the process $X_t'$ to the set $\SSS_1$, we have the coupling epoch.

If this event happens, the processes will stick together.

So, at the end of any regeneration period, the processes $X_t$ and $X_t'$ can coincides with probability $\pi\bd\pi_0\pi_1$ (this probability can be improved, for example, by the choice of $\Theta$).
Therefore the coupling epoch $\tau(X_0')$ is a geometrical sum of the regeneration periods of $X_t'$ including the first incomplete one.

Thus, we can find the upper bounds for $\EE\,[\tau(X_0')]^k$ in the case when $\EE\,\xi_i^k<\infty$ ($C>k-1$):
\vspace{-5mm}
\begin{eqnarray*}
\!\!\!\!\!\!\!\! \EE\,[\tau(X_0')]^k\le \pi\suml_{i=0}^\infty (1-\pi)^i\EE\,\left( R_0+\suml_{j=1}^ib_j\right)^k \le\hspace{4cm}
\\ \\
\le \suml_{i=0}^\infty (1-\pi)^i (i+1)^{k-1}\EE\,\left(R_0^k+\suml_{j=1}^iR_j^k \right)= \hspace{4cm}
\\ \\
=\EE\,R_0^k\times \suml_{i=0}^\infty (1-\pi)^i( i+1)^{k-1} + \EE\,R_1^k\times \suml_{i=0}^\infty (1-\pi)^i (i+1)^{k}=\qquad
\\ \\
=K_0(k,\pi)\EE\,R_0^k+K(k,\pi)\EE\,R_1^k,
\end{eqnarray*}
where $R_0$ is the first regeneration point of $X_t'$, and $R_i$, $i\in \mathbf      N$ are the length of subsequent regeneration periods.

\subsection      {$TV$-distance}
So, if $\PPP_t$ is distribution of $X_t$, and $\PPP_t'$ is distribution of $X_t'$ then
$$
\|\PPP_t-\PPP_t'\|_{TV}\le 2\dd \frac{K_0(k,\pi)\EE\,R_0^k+K(k,\pi)\EE\,R_1^k}{t^k}.
$$

As $\PPP_t'\Longrightarrow \PPP$ for all $X_0'$, then
\begin{eqnarray*}
\|\PPP_t-\PPP\|_{TV}\le 2 \frac{\dd\intl_{\XX}(K_0(k,\pi)\EE\,R_0^k+K(k,\pi)\EE\,R_1^k)\PPP(\ud X_0')}{t^k}=\qquad \qquad \qquad 
\\ \\
=2\; \frac{\dd K_0(k,\pi){\intl_{\XX}\EE\,R_0^k\PPP(\ud X_0')}+K(k,\pi)\EE\,R_1^k}{t^k}.
\end{eqnarray*}
\begin{center}
{\bf But we don't know neither the stationary distribution of $X_t$ nor the stationary distribution for the ``standard'' model!!!}
\end{center}

\subsection{Estimation of stationary distribution}
The lower bounds for the distribution $\mathcal P$ can be obtained by standard methods from a renewal theory (see \cite{Smith}).

$$
\begin{array}{l}
\mathcal P\{n_t=0\}\in[0; 1];
\\ \\
\mathcal P\left\{n_t=m>0, x^{(0)}\le a_1, x^{(1)}\le a_1, x^{(2)}\le a_1, \ldots x^{(m)}\le a_1, \right\}\ge 
\\ \\
\le \dd e^{-\rho}\frac{\rho^k}{k!} \times\frac{\intl_{a_1}^\infty 1-\Psi(u)\ud u}{\intl_0^\infty x^2\ud \Phi(x)}\times \frac{\intl_{a_2}^\infty 1-\Psi(u)\ud u}{\intl_0^\infty x^2\ud \Phi(x)} \times \ldots \frac{\intl_{a_m}^\infty 1-\Psi(u)\ud u} {\intl_0^\infty x^2\ud \Phi(x)}, 
\end{array}
$$
\\
where 
$\Psi(x)=1-\dd\intl_0^x\exp\left(\intl_0^s -Q(s)\ud s\right)\ud x$.

Now, by integration $\dd\intl_{\mathcal X}\EE\,R_0^k\PPP(\ud X_0')$ we obtain the constant $\mathbf K$.

\subsection{Other method to obtain the bounds $\mathbf K$}
Also we can consider an embedded renewal process with the renewal times equal to moments of process returns to the zero-state  $(0,0)$: $\{ t: X_t=(0,0) \}$.
We can find lower bounds for intensity of this flow from formulae (\ref{tak2})--(\ref{razl}) i, and then apply an approach proposed in \cite{zv1,zv2}.

\section*{Acknowledgments}
The author  is grateful to E.~Yu.~Kalimulina for the great help in preparing this paper. The work is supported by RFBR, project No~20-01-00575 A.

\end{document}